\documentclass[12pt]{amsart}
\usepackage{mathpazo,fourier}
\usepackage[english]{babel}
\usepackage[utf8]{inputenc}
\usepackage{graphicx, amssymb, amsmath,amsthm,color}

\def \Sp {\mathbb{S}^1}
\def \vphi {\varphi}
\def \t {\mathsf{t}_*}
\def \x {\mathbf{x}}
\def \z {\mathbold{z}}
\def \tten {\mathsf{T}_{10}}

\newtheorem{thm}{Theorem}
\newtheorem{cor}{Corollary}

\begin{document}
	\author{ Maxim Arnold}
	\address{University of Texas at Dallas, 800, W. 
		Campbell rd., Richardson, TX, 75080, USA}
		\email{maxim.arnold@utdallas.edu}
		\author{Anatoly Eydelzon}
			
	\date{\today}
	
	\title{Finite Gau\ss{} Transformation.}
	\maketitle
	\section{Introduction.}

Usual way to represent real number $\x\in (0,1)$ is via its infinite decimal 
fraction. 
More formally, consider the 
transformation $\tten (\x)=\{10 \x\}=10\x-\lfloor{10\x}\rfloor$ 
where 
$\lfloor{\x}\rfloor$ and $\{\x\}$ 
denote integer 
and fractional parts of 
$\x$ 
correspondingly. Since transformation $\mathsf{Dec}$ is 
ten-folded, to be able to restore $\x$ from its image, one has to get track 
of the lost information. 
Partition the unit interval onto ten parts $(0,1]=\bigcup\limits_{j=1}^{10} 
\left(\frac{j-1}{10},\frac{j}{10}\right]$ and denote $a_n=\lfloor 10 
\tten^{n-1}( 
\x)\rfloor$. In other words,  sequence $\{a_n\}$ encodes the trajectory of 
the point 
$\x$ under the transformation $\mathsf{Dec}$ subject to the above 
partition. Then 
$0.a_1a_2\ldots$ is the decimal fraction representation of $\x$.   
\emph{Convergent} 
$\x_n=0.a_1a_2\ldots a_n$ corresponds to the whole segment of the 
initial points lying in the $10^{-n}$ vicinity since for all these points the 
first $n$ symbols in the encoding of the trajectory coincides. Thus to 
achieve the $\varepsilon$ - accuracy of the approximation by the 
convergents one has to consider $\log \frac{1}{\varepsilon}$ iterations of 
the transformation $\tten$.   

 Another representation can be achieved via famous 
Gauss transformation: $\mathsf{G}\x=\{1/\x\}$. Again, one could 
encode the trajectory of any point $\x\in [0,1]$ 
by the information which is lost at each step. 
Partition the unit interval as 
$(0,1]=\bigcup\limits_{n=1}^{\infty}\left(\frac{1}{n+1},\frac{1}{n}\right]$ 
and denote 
$a_k=\lfloor1/(\mathsf{G}^k\x) \rfloor$. Then the numbers $a_k$ will 
correspond to 
the the elements 
of continued fraction representation
\[\x=\tfrac{1}{a_1+\tfrac{1}{a_2+\tfrac{1}{a_3+\dots}}}.\] 
Convergents obtained by from the first digits of the continuous fraction 
provide the best rational approximation of the point $\x$.  However,  
as a payoff 
one has to use infinite set of symbols to encode the trajectory.

  In the present note we consider the \emph{Finite Gau\ss{} 
  Transformation}. It uses the partition similar to those for the Gau\ss{}, but 
  having only finite number of the intervals. 
  We present invariant density for such transformation and discuss several 
  examples together with periodic orbits.
  
We start with the following transformation of the circle $\Sp$. Fix 
	$n\in \mathbb{N}$ and consider the regular $n$-gon 
	$\mathbf{P}=\{P_1,\dots,P_n\}$. 
	circumscribed 
	around $\Sp$. Denote by $A_j$ the point of tangency of $\Sp$ with the 
	$j$-th side $P_jP_{j+1}$. Points $A_j$ partition 
	$\Sp$ onto $n$ arcs. Transformation $T:\Sp\mapsto
	\Sp$ is 
	defined as follows (see Fig. \ref{fig:Ngon}):
	\begin{equation}
	\label{eq:Finite_Gauss}
	T\z=\Sp\cap P_j\z
	\end{equation}
	\begin{figure}[!h]
		\centering
		\includegraphics[width=0.4\textwidth]{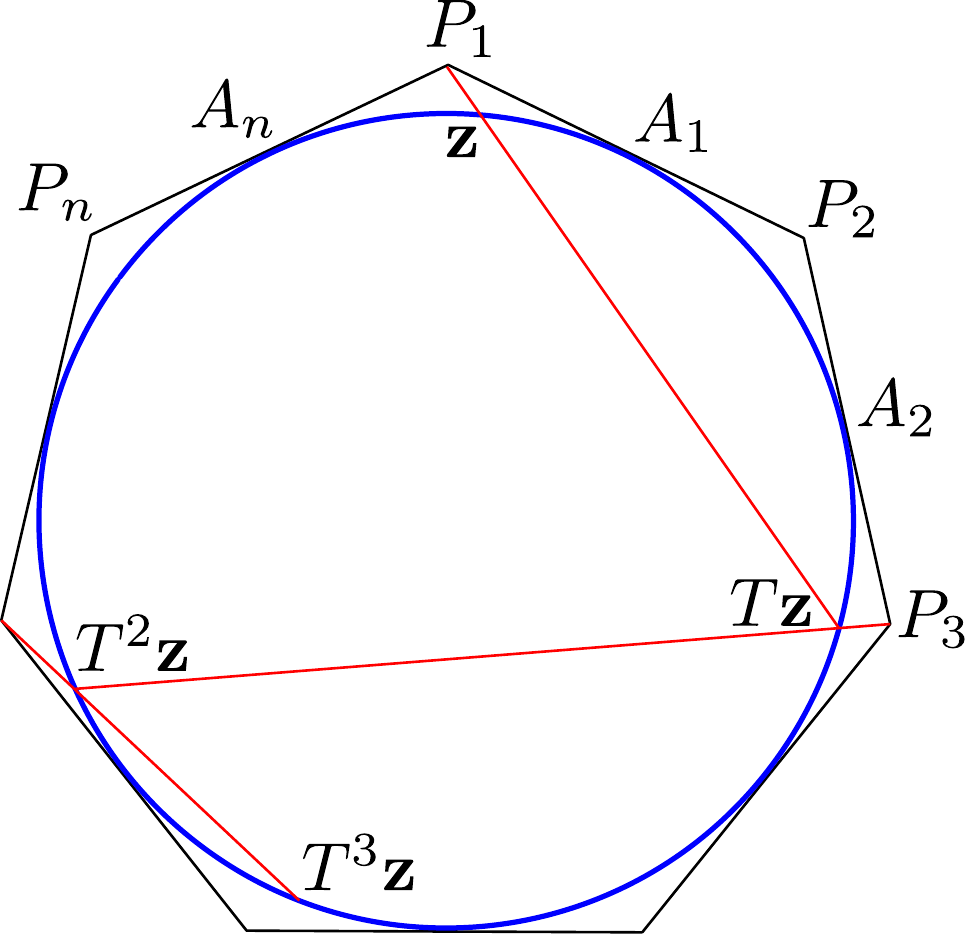}
		\caption{Finite Gau\ss{} transformation.}\label{fig:Ngon}
	\end{figure}
	
Our main result concerns the invariant measure of the  transformation 
\eqref{eq:Finite_Gauss}.  We claim that this invariant measure 
is absolutely continuous with respect to the Lebesgue measure with the 
infinite density  having the form 
\begin{equation}
\label{eq:invariant_density}
\rho(\vphi)=\frac{\sin\frac{\pi}{n}}{\cos\frac{\pi}{n}-\cos\left\{\frac{n\varphi}{2\pi}\right\}}.
\end{equation} In the next section we present a proof of this result and 
discuss its connection to Gau\ss{} transformation.
 

For $n=3$ transformation \eqref{eq:Finite_Gauss} was considered in 
\cite{Bogoyavlenski}, \cite{Bogoyavlenski_book} as describing the mass 
evolution in the cosmological models of Bianchi type. Amazingly, this
transformation for $n=4$ has been considered 
only thirty years later describing the dynamics of Pythagorean triples (see 
\cite{Romik}).  
\subsection*{Acknowledgments.} Authors are deeply indebted to Ya.G. 
Sinai , 
O.I. Bogoyavlenski, Yu.M. Baryshnikov  and S.L. Tabachnikov for fruitful 
discussions and constant support.
	
	\section{Computations.}
	
	We start with parametrization of $\Sp$  by the stereographic projection  
	from 
	$\mathbb{R}$ to $\Sp$. The 
	coordinates of 
	a point $\z\in \Sp$ 
	are given by 
	$\z=\left(\frac{2t}{1+t^2},\frac{1-t^2}{1+t^2}\right)$ and 	point $A_1$  
	correspond to the value $\t=\tan \frac{\pi}{2n}$. Vertex $P_1$ has 
	coordinates $\left(0,\frac{1}{\cos\frac{\pi}{n}}\right)$. 

	Transformation \eqref{eq:Finite_Gauss} is defined as
	\begin{equation}
	\label{eq:transform_gen}
	T\z=\z+\tau (\z-P_1), \qquad \tau =\frac{2\z\cdot 
	(P_1-\z)}{|\z-P_1|^2}.
	\end{equation}
	Since $\t$ is the tangent 
	of 
		the half of $\frac{\pi}{n}$, we get the expressions
		$\cos\frac{\pi}{n}=\frac{1-\t^2}{1+\t^2}$ and 
		$\sin\frac{\pi}{n}=\frac{2\t}{1+\t^2}$. Therefore $\frac{1}{\cos 
			\frac{\pi}{n}}=\frac{1+\t^2}{1-\t^2}$ and so 
			
			$$\tau=\dfrac{\left(\t^2-1\right)\left(t^2-\t^2\right)}{t^2+\t^4}.$$

Immediate computation shows  that 
$T\z=\left(\frac{2s}{1+s^2},\frac{1-s^2}{1+s^2}\right)$ where 		
\begin{equation}\label{eq:SfromT}
		s=Q(t)=\dfrac{\t^2}{t}:[-\t,\t]\mapsto [-\t,\t]^c.
		\end{equation}
		
		\begin{figure}[hbt]
			\centering
			\includegraphics[width=0.4\textwidth]{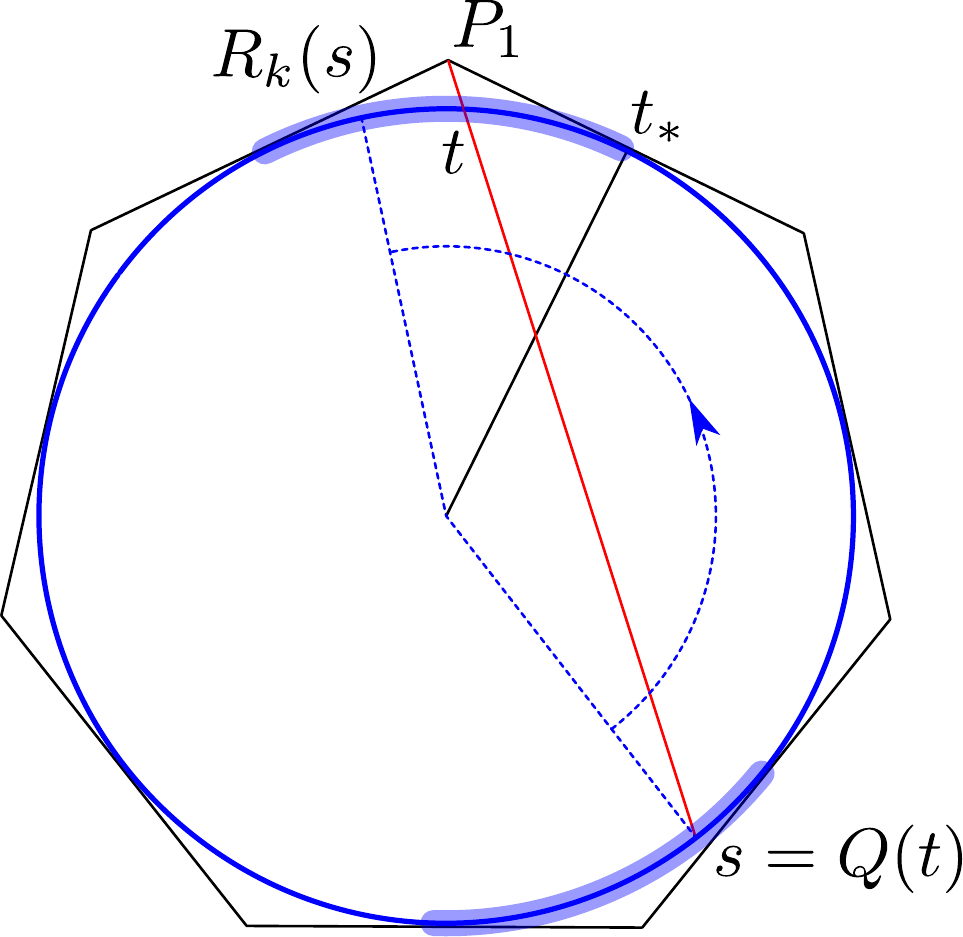}
			\caption{Transformation on the segment $[-\t,\t]$ is the 
			composition 
				of the transformation $Q$ and rotation $R_k$.}
		\end{figure}
		
	In order to iterate our transformation we have to rotate the image $T\z$ 
	back to the original interval $[-\t,\t]$.
	by the rotation 
	
	\begin{equation}
	\label{eq:s_rotation}R_ks= \dfrac{a_{k,n} s-b_{k,n}}{b_{k,n}s 
	+a_{k,n}},\quad \mbox{ for } s\in 
		\left[\tan\frac{\pi(2k-1)}{2n}, 
		\tan 
	\frac{\pi(2k+1)}{2n}\right].
	\end{equation}
	here $a_{k,n}=\cos 
	\frac{\pi k}{n}$ and $b_{k,n}=\sin \frac{\pi k}{n}$.
	Thus we can associate $T$ with the transformation $\mathsf{F}$ of the 
	interval 
	$[-\t,\t]$ having the form 
	\begin{equation}
	\label{eq:F_k}
	\begin{aligned}
	\mathsf{F}(t)=&(R_k\circ Q) (t)=\frac{-b_{k,n} t+a_{k,n} 
		\t^2}{a_{k,n} t+b_{k,n}\t^2}, \\ &\mbox{
		for }\,
	%
	t\in \left[Q^{-1} \left(\tan\tfrac{\pi(2k+1)}{2n}\right), 
	Q^{-1}\left(\tan\tfrac{\pi(2k-1)}{2n}\right)\right]
	\end{aligned}
	\end{equation}

	\begin{thm}
	Transformation \eqref{eq:F_k} has invariant measure with the 
	density $\rho(t)=\dfrac{2\t}{(t^2-\t^2)}$.
	\end{thm}
	\begin{proof}
		
		We have to show that function $\rho$ satisfies the equation
		\begin{equation}
		\label{eq:inv_rho}
		\rho(t)=\sum\limits_{k=1}^{n-1}\rho\left(F_{k}^{-1}(t)\right)
		\left|\left(F_k^{-1}\right)'\right|
		\end{equation}
		
		Since $\rho(t)=\dfrac{1}{t-\t}-\dfrac{1}{t+\t}$ we have to check that 
		\[ \dfrac{1}{t-\t}-\dfrac{1}{t+\t}=\sum\limits_{k=1}^{n-1}\left(
		\frac{\left|\left(F_k^{-1}\right)'\right|}{F_k^{-1}(t)-\t}-
		\frac{\left|\left(F_k^{-1}\right)'\right|}{F_k^{-1}(t)+\t}\right)=\left(\ln
		\prod\limits_{k=1}^{n-1} \frac{F^{-1}_k(t)-\t}{F_k^{-1}(t)+\t}\right)'
		 \]
		 
		 Introducing the expression for the preimage
		 \[F_k^{-1}=\dfrac{\t^2(b_{k,n}t-a_{k,n})}{-a_{k,n} t-b_{k,n}}\]

and denoting $\tan \frac{\pi k}{n}$ by $\mathrm{t}_{n,k}$, we get

$$\prod\limits_{k=1}^{n-1} 
\dfrac{F^{-1}_k(t)-\t}{F_k^{-1}(t)+\t}=\prod\limits_{k=1}^{n-1} 
\dfrac{\t(1- t \mathrm{t}_{n,k})+(t+\mathrm{t}_{n,k})}
{\t(1-t \mathrm{t}_{n,k})-(t+\mathrm{t}_{n,k})}=$$

$$=\prod\limits_{k=1}^{n-1} -\frac{t (1-\t 
\mathrm{t}_{n,k})+(\mathrm{t}_{n,k}+\t)}{t (1+\t 
\mathrm{t}_{n,k})+(\mathrm{t}_{n,k}-\t)}.$$
	Since $\t=\tan\frac{\pi}{2n}$ we can rewrite $\mathrm{t}_{n,k}\pm \t$ 
	as $\mathrm{t}_{2n,2k\pm 1}(1\mp \t \mathrm{t}_{n,k})$ and so

$$\prod\limits_{k=1}^{n-1} 
\dfrac{F^{-1}_k(t)-\t}{F_k^{-1}(t)+\t}=\prod\limits_{k=1}^{n-1} 
\dfrac{1-\t \mathrm{t}_{n,k}}{1+\t\mathrm{t}_{n,k}} 
\prod\limits_{k=1}^{n-1} 
-\dfrac{t+\mathrm{t}_{2n,2k+1}}{t+\mathrm{t}_{2n,2k-1}}.$$

The first product equals to one, since 
$\mathrm{t}_{n,k}=-\mathrm{t}_{n,n-k}$.
The last product is telescopic and so it is equal to $\dfrac{t-\t}{t+\t}$. 
Since $\left(\ln \dfrac{t-\t}{t+\t}\right)'=\rho(t)$ the theorem is proven. 
\end{proof}

\begin{cor}
	Original density equals 
	$\rho(\varphi)=\frac{\sin\frac{\pi}{n}}{\cos\frac{\pi}{n}-\cos\vphi
	}.$
	\end{cor}

 \section{Discussion.}
 
 \subsection{Fixed points.}
 
 Transformation \eqref{eq:F_k} has $n-1$ fixed points. If $s_k$ is the 
 $k$-th fixed point of \eqref{eq:F_k} then $T s_k=e^{2\pi i k/n} 
 s_k$.Therefore $s_k$ correspond to the 
 periodic point of  
 the transformation \eqref{eq:Finite_Gauss} of period $\gcd(k,n)$ (see 
 Fig.\ref{fig:triangle}). 
 \begin{figure}[hbt]
 	\centering
 	\includegraphics[width=0.4\textwidth]{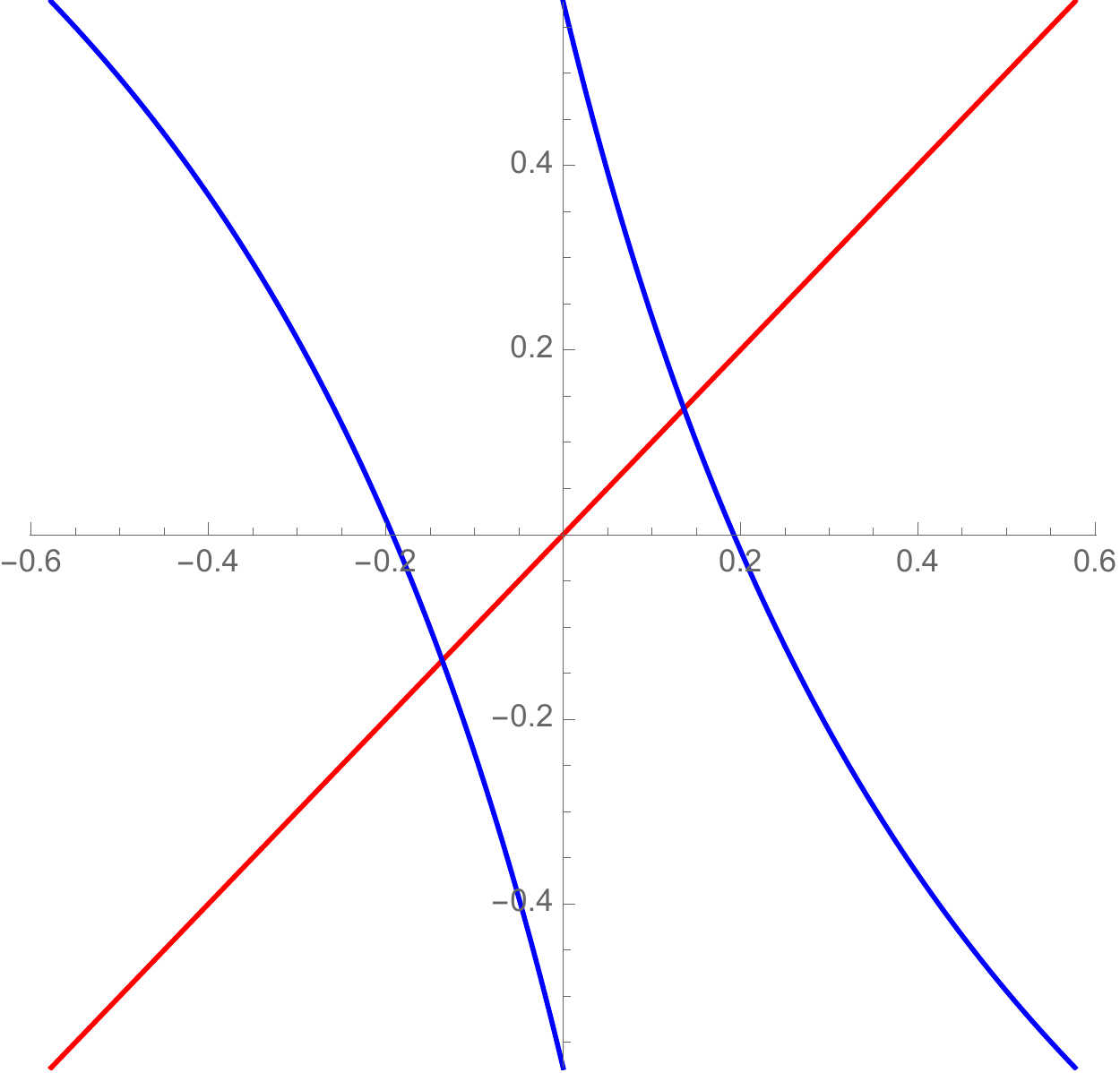}\qquad
 	\includegraphics[width=0.5\textwidth]{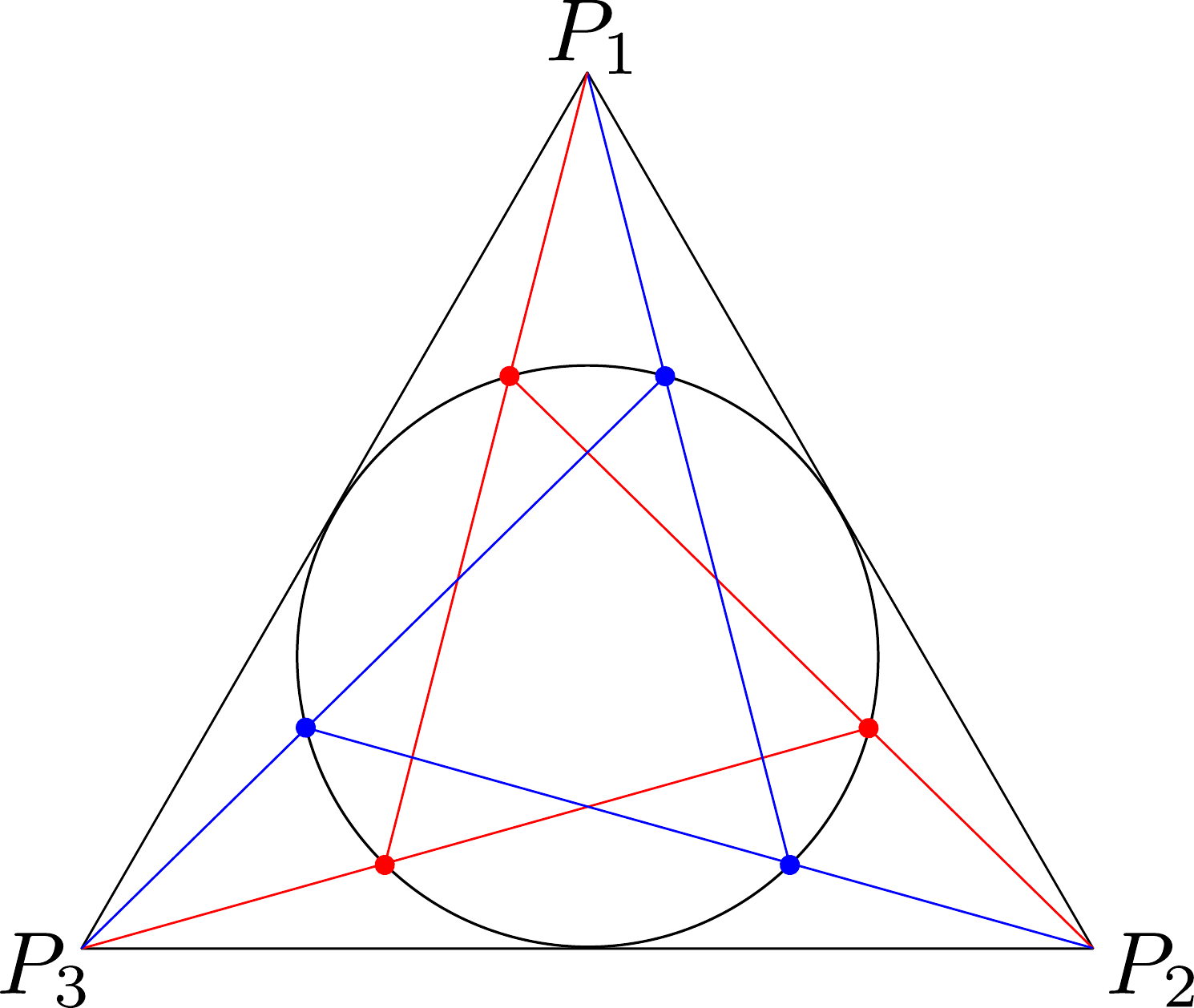}
 	\caption{Fixed point of Transformation 
 		\eqref{eq:F_k}.}\label{fig:triangle}
 \end{figure}

 If one orients the polygon $\mathbf{P}$ in such a way that the point $A_n$ 
 will belong to the vertical axis, then the arc $A_nA_1$ will correspond to 
 the interval $[0,\tan\frac{\pi}{n}]$ in $t$ variable. Then points 
 $Q^{-1}(A_k)$ will correspond to the values 
 $t_k=\frac{1}{k}\tan\frac{\pi}{n}$. 
 \begin{figure}[hbt]
 	\centering\includegraphics[width=0.4\textwidth]
 	{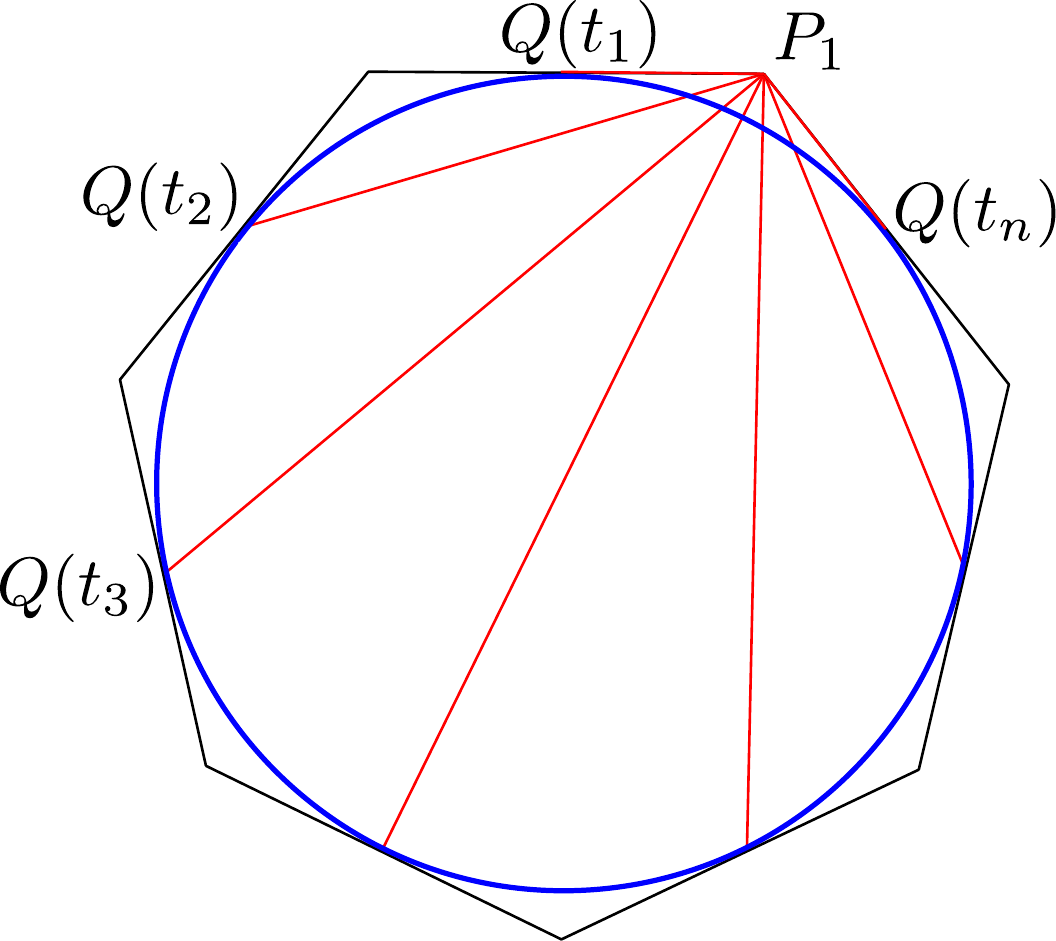}
 	\caption{More on Gau\ss{}}
 \end{figure}
 Thus, for example for $n=4$ points of 
 tangency will split the interval $[0,1]$ on three intervals $[0,\frac 
 13]\cup[\frac 13,\frac 12]\cup[\frac 12,1]$. And the corresponding 
 transformation will have the form
 \begin{equation}
 \label{eq:square}
 T(t)=\begin{cases}
 \frac{3t-1}{t-1},&t\in[0,\frac 13]
 \\
 \frac{1}{t}-2,&t\in[\frac 13,\frac 12]\\
 \frac{t-1}{1-3t},& t\in [\frac 12,1]
 \end{cases}
 \end{equation} (compare with \cite{Romik})
Obviously, in this orientation any periodic point of the transformation 
$\mathsf{F}$ is the periodic point of some Moebius transformation with 
integer 
coefficients. Therefore, every periodic point of $\mathsf{F}$ is a a root of 
some 
quadratic polynomial. However, for the moment we do not know if any 
quadratic irrationality  can be achieved as a periodic point of 
transformation $\mathsf{F}$.

 \subsection{Piecewise Moebius on the interval}
 
 Transformation \eqref{eq:transform_gen} for the case $n=3$ can be 
 considered not only for regular triangle. Projectively mapping generic 
 triangle to the regular one, we will obtain the transformation 
 \eqref{eq:F_k}  of the form
 $$T_a t=\begin{cases}
 \frac{at+1}{(a-2)t-1},\qquad t\in[-1,0]\\
 \frac{at-1}{-(a-2)t-1},\qquad t\in [0,1]
 \end{cases}$$ 
It is easy to see that such transformation has invariant density 
$\rho(t)~=~\frac{1}{t^2-1}$. It is still not clear how does finite Gauss 
transform relate to the other known linear fractional transformations of the 
interval (see \cite{Grochenig} and references therein.)

\subsection{Higher dimensions.}
 It would be interesting to look at the $3$-dimensional analog of 
 transformation $\mathsf{F}$. Given vertices of the tetrahedron, provide the 
 partition of the sphere on the four regions, defined by the perpendicular 
 bisector planes to the edges of the tetrahedron. Each region correspond to 
 the nearest vertex. Thus one can define the transformation 
 via formulas \eqref{eq:Finite_Gauss}.
 
 \begin{figure}[hbt]
 	\centering\includegraphics[width=0.5\textwidth]{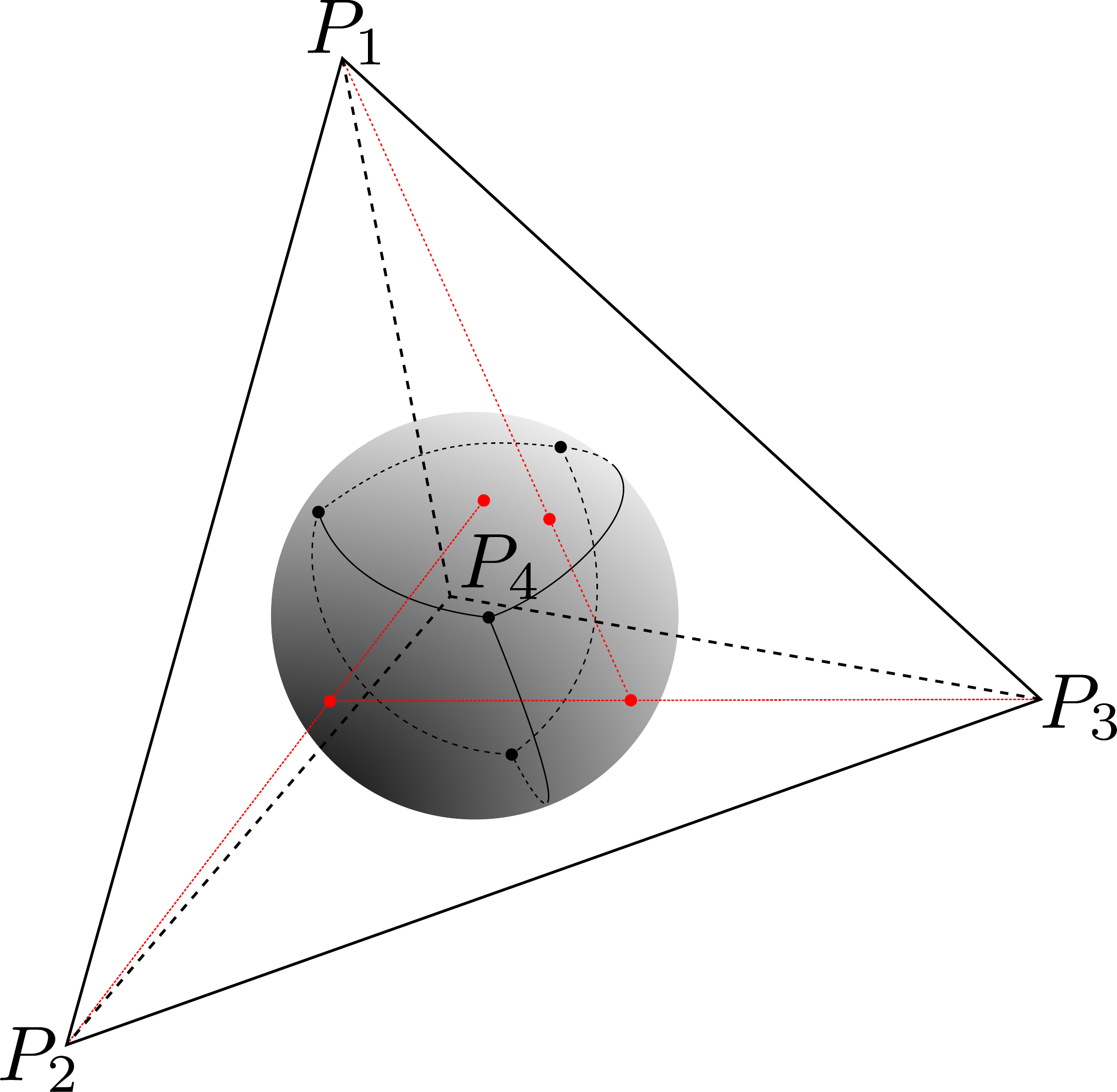}
 	\caption {$3d$ Finite Gauss map.}
 \end{figure}

For the sphere we have the following parametrization: 
$$\z=\left(\frac{2u}{1+u^2+v^2},\frac{2v}{1+u^2+v^2},
\frac{1-u^2-v^2}{1+u^2+v^2}\right)$$
and $P_1=(0,0,3)$. Therefore from \eqref{eq:transform_gen} we compute 
for $z$ from the top part of the sphere (see Fig. 5)
$Tz=Q(u,v)$ where $$
Q(u,v)=\left(\dfrac{u}{2(u^2+v^2)},\dfrac{v}{2(u^2+v^2)}\right).$$
 
 Our experiments shows that the invariant density is infinite at the points of 
 tangency of the tetrahedron and the sphere and is smooth everywhere else. 
 However, we still do not know 
 the closed formula for it.

		\bibliographystyle{plain}
	\bibliography{sinnov}
\end{document}